\documentclass[12pt]{article}
\usepackage{amsmath}
\usepackage{amssymb}
\usepackage{amsthm}
\usepackage{bbm}
\usepackage[all]{xy}

\newtheorem{Thm}{Theorem}

\newtheorem{theorem}{Theorem}[section]
\newtheorem{lemma}[theorem]{Lemma}
\newtheorem{corollary}[theorem]{Corollary}

\newtheorem{proposition}[theorem]{Proposition}

\theoremstyle{definition}
\newtheorem{remark}[theorem]{Remark}
\newtheorem{definition}[theorem]{Definition}

\def\Z{\mathbb Z}
\def\C{\mathbb C}

\def\0{\underline 0}

\begin{document}

\title {\bf On the Milnor classes of local complete intersections\footnote{Research partially supported by CAPES, CNPq and FAPESP,
Brazil, and by CONACYT and PAPIIT-UNAM, Mexico. \newline $\quad$ {\it
Key-words:} Complete intersections, Milnor classes,
Whitney stratifications, Schwartz-MacPherson classes, Fulton-Johnson classes.
\newline   {\it Mathematics Subject Classification.} Primary ;
32S50, 32S15, 14B05 Secondary; 14C17, 14J17, 55S35  }}

\vspace{1cm}
\author{R. Callejas-Bedregal, M. F. Z. Morgado and J. Seade }

\date{\bf }
\maketitle

\begin{abstract}
In this work we study algebraic, geometric and topological
properties of  the Milnor classes of local complete
intersections
 with arbitrary singularities. We describe first the
Milnor class of the intersection of a finite number of
hypersurfaces, under certain conditions of transversality, in
terms of the Milnor classes of the  hypersurfaces. Using this
description we obtain a Parusi\'{n}ski-Pragacz type formula, an
Aluffi type formula and a description of the Milnor class of the
local complete intersection in terms of the global L\^e cycles of
the hypersurfaces that define it. We consider next  the general
case of a local complete intersection $Z(s)$ defined by a regular
section $s$ of a rank $r$ holomorphic bundle $E$ over a compact
manifold $M$, $r \geq 2$. We notice that  $s$ determines a
hypersurface $Z(\tilde s)$ in the total space of the
projectivization $\mathbb{P}(E^{\vee})$ of the dual bundle
$E^{\vee}$, and we give a formula  expressing
 the total Milnor class of the local complete intersection $Z(s)$ in terms of the
 Milnor classes of the  hypersurface $Z(\tilde s)$.

\end{abstract}
\par\medskip
\noindent{\footnotesize\textbf{Note added in 2026:} This 2012 preprint gave rise to two separate papers:
(i) ``On the Chern classes of singular complete intersections'', arXiv:1909.01117,
published in Journal of Topology 13 (2020), 159--174;
(ii) ``The Cayley-Bedregal reduction for Milnor and Hirzebruch-Milnor classes'',
arXiv:2611.13267 (2026).}\par\medskip

\section*{Introduction}

If $Y$ is a singular variety in a complex manifold $M$, its total Milnor class
${\cal M}_*(Y)$ is an element in its Chow group $A_*(Y)$ that measures the
difference between its total Schwarz-MacPherson class $c_*^{SM}(Y)$ and its
total Fulton-Johnson class $c_*^{FJ}(Y)$ of $Y$. Both of these classes are
generalizations for singular varieties of the classical Chern classes of
manifolds. The total Milnor class actually has support in the singular set
${\rm Sing}(Y)$, and there is a Milnor class in each dimension, from $0$ to
that of ${\rm Sing}(Y)$. In particular, when $Y$ has only isolated
singularities which are all local complete intersections, there is only a
$0$-degree Milnor class, which is an integer, and this is the sum of the
local Milnor numbers (by \cite{SS,Suwa}).

The concept of Milnor classes appears first implicitly in P. Aluffi's work
\cite{Aluffi,Aluffi1} on $\mu$-classes for hypersurfaces in algebraic
manifolds. Milnor classes for hypersurfaces also appear implicitly in
A. Parusi\'{n}ski and P. Pragacz' article \cite{Par-Pr0}, though the actual
name of Milnor classes was coined later by various authors at about the same
time (see \cite{BLSS1,BLSS2,Yokura,Par-Pr}). The case of local complete
intersections was first envisaged in \cite{BLSS1,BLSS2}.

The study of Milnor classes is an active field of current research, with
significant applications to other related areas (see for instance
\cite{Alu-Mar,Beh,BSY,Max}). Furthermore, this notion is notably extended
in \cite{CMSS} to that of Milnor-Hirzebruch classes of locally complete
intersections in algebraic manifolds. These are Hodge-theoretic classes that
arise from the Hirzebruch classes introduced by Brasselet et al in
\cite{BSY} (see also \cite{MSS}).

Milnor classes are mysterious ``objects'' that encode much information about
the varieties in question, and this is being studied by various authors from
several points of views. For instance, when $Y$ is defined by a regular
section of a very ample holomorphic line bundle over $M$, it is proved in
\cite{BMS} that the Milnor classes determine the global L\^e cycles and
vice versa. The global L\^e cycles are a natural extension of the local
L\^e cycles introduced by D. Massey in \cite{Massey0,Massey} for holomorphic
map-germs, which determine (among other things) the topology of the local
Milnor fibre up to homeomorphism.

Yet, most of the work on Milnor classes in the literature is for the case of
hypersurfaces. Here we focus on the complete intersection case, which is
much harder ({\it cf.} \cite{BLSS1,BLSS2,MSS}).

We consider first a finite collection $\{E_i\}$, $r\geq 2$, of holomorphic
vector bundles $E_i$ over $M$ of rank $d_i$. For each of these bundles we
consider a regular holomorphic section $s_i:M\rightarrow E_i$ and let $X_i$
be the $(n-d_i)$-dimensional local complete intersection defined by the
zeroes of $s_i$. We assume further that the $X_i$ are equipped with Whitney
stratifications such that all the intersections amongst strata in the various
$X_i$ are transversal. Then we prove a result (Theorem \ref{4.1}) that
describes the total Milnor class of $X$ in terms of the total
Schwartz-MacPherson and Milnor classes of the $X_i$.
Our starting point is the work by Ohmoto and Yokura in \cite{O-Y}, describing
the total Milnor class of finite Cartesian products of hypersurfaces. As a
consequence we get (Corollary \ref{11}):

\begin{Thm}
With the above hypothesis we have:
\[
{\cal M}(X)=(-1)^{r-1}
 c\left(\left(TM|_X\right)^{\oplus r-1}\right)^{-1}
\cap\sum_{i=1}^{r} a_{1,i}\cdot\ldots\cdot a_{r-1,i}\cdot{\cal M}(X_i),
\]
where
\[
a_{j,i}=\begin{cases}
c^{Vir}(X_{j+1}) &\text{if }j\leq i,\\
c^{SM}(X_j) &\text{if }j>i.
\end{cases}
\]
\end{Thm}

We focus next on the case when the bundles in question are all line bundles
$L_i$, so each $X_i$ is a hypersurface. We get three types of applications:

\noindent {\bf i)} A Parusi\'{n}ski-Pragacz type formula for local complete
intersections as above (Corollary \ref{P-P}). This answers positively the
expected description given by Ohmoto and Yokura in \cite{O-Y} for the total
Milnor class of a local complete intersection, as a polynomial in the Chern
classes $c_1(L_i)$, $i=1,\ldots,r$. We notice that a remarkable (different)
generalization of the Parusi\'{n}ski-Pragacz formula for complete
intersections has been given recently in \cite{MSS}.

\noindent {\bf ii)} A link between the Milnor classes of a local complete
intersection $X$ as above and the global L\^e cycles of the hypersurfaces
$X_i$ that define it (Theorem \ref{leformula}).

\noindent {\bf iii)} A description of the total Milnor class of the local
complete intersection $X$ in the vein of Aluffi's formula in \cite{Aluffi}
for hypersurfaces, using Aluffi's $\mu$-classes (Corollary \ref{Aluffi}).

Finally we look at the general case: We consider a holomorphic vector
bundle $E$ of rank $r$ over an $n$-dimensional compact complex analytic
manifold $M$, and we let $Z(s)$ be the zero set of a regular holomorphic
section $s$ of $E$. Then $Z(s)$ is an $(n-r)$-dimensional local complete
intersection. We look at the corresponding projectivized bundle
$\mathbb{P}(E^{\vee})$ and notice that the section $s$ induces a section
$\tilde{s}$ of the tautological bundle
${\cal O}_{\mathbb{P}(E^{\vee})}(1)$ over $\mathbb{P}(E^{\vee})$,
and therefore it defines a hypersurface $Z(\tilde{s})$ in
$\mathbb{P}(E^{\vee})$. Then we give a formula expressing the total
Milnor class of the local complete intersection $Z(s)$ in terms of the
total Milnor class of the hypersurface $Z(\tilde{s})$ and the Chern
classes of the various bundles in question. We prove (Theorem \ref{t:final}):

\begin{Thm}
Let $p:\mathbb{P}(E^{\vee})\rightarrow M$ be the projectivization of
the dual bundle $E^{\vee}$, so we have the tautological exact sequence
\[
0\rightarrow F\rightarrow p^*E\rightarrow
{\cal O}_{\mathbb{P}(E^{\vee})}(1)\rightarrow 0.
\]
Then, setting ${\cal O}(1):={\cal O}_{\mathbb{P}(E^{\vee})}(1)$, we get:
\[
{\cal M}(Z(s))=p_*\left(\left[
c\big(p^*(E^{\vee})\otimes{\cal O}(1)\big)^{-1}
\cdot c_1\big({\cal O}(1)\big)^{r-1}\cdot c(F)^{-1}\cdot
c_{top}(F)\right]\cap{\cal M}(Z(\tilde{s}))\right),
\]
where $\tilde{s}$ is the section induced by $p^*s$ in ${\cal O}(1)$.
\end{Thm}

We are grateful to Nivaldo Medeiros Jr., from Universidade Federal
Fluminense in Brazil, for fruitful suggestions that helped us a lot to get
our results in final form. We are also grateful to Marcelo Jos\'e Saia,
from USP at S\~ao Carlos-Brazil, for helpful conversations.

\section{Derived Categories}

\hspace{0,5cm} We  assume some basic knowledge on  derived categories,
hypercohomology and  sheaves of vanishing cycles as described in
\cite{Dimca}.

If $X$ is a complex analytic space then ${\cal D}^{b}_{c}(X)$
denotes the derived category of bounded, constructible complexes of
sheaves of $\C$-vector spaces on $X$. We denote the objects of
${\cal D}^{b}_{c}(X)$ by something of the form $F^{\bullet}$. The
shifted complex $F^{\bullet}[l]$ is defined by
$(F^{\bullet}[l])^{k}=F^{l+k}$ and  its differential is
$d^{k}_{[l]}=(-1)^{l}d^{k+l}$. The constant sheaf $\C_{X}$ on $X$
induces an object $\C_{X}^{\bullet} \in {\cal D}^{b}_{c}(X)$ by
letting $\C_{X}^{0}=\C_{X}$ and $\C_{X}^{k}=0$ for $k\neq 0$.

If $h:X\rightarrow \C$ is an analytic map and $F^{\bullet}\in
{\cal D}^{b}_{c}(X)$, then we denote the sheaf of vanishing cycles
of $F^{\bullet}$  with respect to $h$ by $\phi_{h}F^{\bullet}$.

For $F^{\bullet}\in {\cal D}^{b}_{c}(X)$ and $p \in X$, we denote
by ${\cal H}^{*}(F^{\bullet})_{p}$ the stalk cohomology of
$F^{\bullet}$ at $p$, and by $\chi(F^{\bullet})_{p}$ its Euler
characteristic. That is,
$$\chi(F^{\bullet})_{p}=\sum_{k}(-1)^{k}{\rm dim}_{\C}{\cal
H}^{k}(F^{\bullet})_{p}.$$

We also denote by $\chi(X,F^{\bullet})$ the Euler characteristic of
$X$ with coefficients in $F^{\bullet}$, {\it i.e.},
$$\chi(X,F^{\bullet})=\sum_{k}(-1)^{k}{\rm dim}_{\C}\mathbb{H}^{k}(X,F^{\bullet}),$$
where $\mathbb{H}^{*}(X,F^{\bullet})$ denotes the hypercohomology
groups of $X$ with coefficients in $F^{\bullet}$.

When $F^{\bullet}\in {\cal D}^{b}_{c}(X)$ is ${\cal
S}$-constructible, where ${\cal S}$ is a Whitney stratification of
$X$, we  denote it by $F^{\bullet}\in {\cal D}^{b}_{{\cal S}}(X)$.
We have
\cite[Theorem 4.1.22]{Dimca}:
\begin{equation}\label{EulerCharact}\chi(X,F^{\bullet})=\sum_{S\in {\cal
S}}\chi(F^{\bullet}_{S})\chi(S),\end{equation} where
$\chi(F^{\bullet}_{S})=\chi(F^{\bullet})_{p}$ for an arbitrary point
$p \in S$.

For a subvariety $X$ in a complex manifold  $M$ we denote its conormal variety by
$T^{*}_{X}M$. That is, $$T^{*}_{X}M:={\rm closure}\;\{ (x, \theta) \in T^{*}M\;|\; x
\in X_{{\rm reg}}\;{\rm and}\; \theta_{|_{T_{x} X_{{\rm
reg}}}}\equiv 0\}\,,$$ where $T^{*}M$ is the cotangent bundle of $M$ and
$X_{{\rm reg}}$ is the regular part of $X$.

\ The following definition is standard in the literature (see
\cite{GM}):

\begin{definition}
Let $ X $ be an analytic subvariety of a complex manifold $ M $, $\{
S_{\alpha} \}$ a Whitney stratification of $M$ adapted to $ X $ and
$x\in S_\alpha$ a point in $X$. Consider $g:(M,x)\rightarrow (\C,0)$
a germ of holomorphic function such that $d_{x}g $ is a {\it
non-degenerate covector} at $x$ with respect to the fixed
stratification. That is, $d_{x}g \in T^{*}_{S_\alpha}M$ and $d_{x}g
\not\in T^{*}_{S^{'}}M$  for all stratum $S^{'} \neq S_\alpha$.  And
let $N$ be  a germ of a closed complex submanifold of $M$ which is
transversal to $S_\alpha$, with $N \cap S_\alpha=\{ x\}$. Define the
{\it complex link }  $l_{S_\alpha}$  of $S_\alpha$ by:
 \begin{center}$l_{S_\alpha}:= X\cap N \cap
B_{\delta}(x)\cap \{g=w\}\quad{\rm for}\;0<|w|<\!\!< \delta<\!\!<
1.$\end{center} The {\it normal Morse datum} of $S_\alpha$ is
defined by:
$$NMD(S_\alpha):=(X\cap N \cap B_{\delta}(x),l_{S_\alpha}),$$ and the {\it normal Morse index} $\eta(S_\alpha,F^{\bullet})$  of the stratum is:
\begin{center}$\eta(S_\alpha,F^{\bullet}):=\chi(NMD(S),F^{\bullet}),$\end{center}
where the right-hand-side means the Euler characteristic of the
relative hypercohomology. \end{definition}

By a result of M. Goresky and R. MacPherson in \cite[Theorem 2.3]{GM}  we
get that the number $\eta(S_\alpha,F^{\bullet})$ does not depend on
the choices of $x\in S_\alpha,\; g$ and $N$.

Notice that by \cite[Remark 2.4.5(ii)]{Dimca}, it follows that
\begin{equation}\label{RelativeEuler}\eta(S_\alpha,F^{\bullet})=\chi(X\cap N \cap
B_{\delta}(x),F^{\bullet})-\chi(l_{S_\alpha},F^{\bullet})\,.\end{equation}

\begin{remark}\label{ConstructibleFunction} Everything we have defined so far
for a constructible complex of sheaves is defined by J. Sch\"{u}rmann
and  M. Tib\u ar in \cite{ST} for constructible functions, and the two
 constructions are somehow equivalent. In fact, given
$F^{\bullet}\in {\cal D}^{b}_{c}(X)$, we have naturally associated
the  constructible function on $X$ given by
$$\beta(p)
=\chi(F^{\bullet})_{p}.$$ Moreover the converse also holds (see \cite{S}), {\it i.e.}, given any constructible
function $\beta$ on $X$ there is $F^{\bullet}\in {\cal
D}^{b}_{c}(X)$ such that
$$\beta(p) =\chi(F^{\bullet})_{p}.$$
\end{remark}

\section{Milnor classes}

\hspace{0.5cm} Let $M$ be an $n$-dimensional compact complex
analytic manifold and let $E$ be a holomorphic vector bundle over
$M$ of rank $d$. Let $Y$ be the zero set of a regular holomorphic
section of $E$, which is an $(n-d)$-dimensional local complete
intersection. Consider the virtual bundle $\tau(Y;M):= T
M|_{_{Y}}-E|_{_{Y}}$, where $T M$ denotes the tangent bundle of $
M$ and the difference is in the  KU-theory of $Y$. The virtual
homology class of $Y$ is defined by the Chern class of  $\tau(Y;M)$
via the Poincar\'e morphism, that is,
$$c^{Vir}(Y;M)=c(\tau(Y;M))\cap
[Y]:=(c^{}(T M|_Y) c^{}(E|_Y)^{-1})\cap [Y].$$
 Notice that restricted to
the regular part of $Y$, the bundle $E$ is isomorphic to the normal
bundle of $Y$ in $M$. Notice also that in the case we envisage here,
these classes coincide with the Fulton-Johnson classes (see for
instance \cite{Ful}).

When there is no ambiguity, for simplicity we will denote the virtual bundle and the virtual classes simply by $\tau(Y)$ and $c^{Vir}(Y)$.

On the other hand, consider the Nash blow up $\tilde{Y} \stackrel{\nu}{\rightarrow}
 Y$ of  $Y$, its Nash bundle ${\tilde T} \stackrel{\pi}{\rightarrow}
\tilde{ Y} $ and the Chern classes of $\tilde{T}$, $c^{j}(\tilde{T})
\in H^{2j}(\tilde{ Y})$, $j=1,\ldots,n$. The Mather classes of $Y$ are defined by
$$c^{Ma}_{k}( Y):=v_{*}(c^{n-d-k}(\tilde{T})\cap [\tilde{ Y}])\in H_{2k}( Y),\;\;k=0,\ldots,n \,.$$
We equip $Y$ with a Whitney stratification $Y_{\alpha}$.
The MacPherson classes are obtained from the Mather classes by considering
appropriate ``weights" for each stratum, determined by the local
Euler obstruction ${\rm Eu}_{{ Y_{\alpha}}}(x)$.  This is an integer associated   in \cite{MacP}
to each point $x \in Y_{\alpha}$.  It is proved in \cite{MacP} that
there exists a unique set of integers $b_{\alpha}$,
 for which the
equation $\sum b_{\alpha} {\rm Eu}_{\bar{ Y}_{\alpha}}(x)=1$ is
satisfied for all points $x \in  Y$. Here, $\bar{ Y}_{\alpha}$ denotes the closure of the stratum, which is itself analytic and therefore it has its own local
Euler obstruction ${\rm Eu}_{\bar{ Y}_{\alpha}}$ and Mather classes, and the sum runs over all
strata $ Y_{\alpha}$ containing $x$ in their closure.

Then the  MacPherson class of degree $k
$ is defined by
$$c^{M}_{k}( Y):=\sum
b_{\alpha}\;i_{*}(c^{Ma}_{k}(\bar{ Y}_{\alpha})),$$
where  and
$i:\bar{ Y}_{\alpha}\hookrightarrow  Y$ is the inclusion map.

We remark that by \cite{BS}, the MacPherson classes coincide,
up to Alexander duality, with the classes defined previously by
M.-H. Schwartz in \cite{Sch}. Thus, following the modern
literature (see for instance \cite{Par-Pr, BLSS2, BSS}), we call these the
Schwartz-MacPherson classes of $Y$ and  denoted
 by $c^{SM}_{k}(Y)$.

\begin{definition}
The Milnor class of $Y$ is:
$${\cal M}(Y):=(-1)^{n-d}\left(c^{Vir}(Y)-c^{SM}(Y) \right).$$
\end{definition}

\section{Milnor classes and the diagonal embedding}

Let $M$ be as before, an $n$-dimensional compact complex analytic manifold,
and set $M^{(r)}:=M\times\ldots\times M$. In this section we let $E$ be
a holomorphic vector bundle over $M^{(r)}$ of rank $d$. Consider
$\Delta:M\rightarrow M^{(r)}$ the diagonal morphism, which is a regular
embedding of codimension $nr-n$. Following \cite[Chapter 6]{Ful} we have
that $\Delta$ induces the Gysin homomorphism
\[
\Delta^!: A_k(M^{(r)})\rightarrow A_{k-nr+n}(M),
\]
given by $\Delta^!(\alpha_1\times\ldots\times\alpha_r)
=\alpha_1\cdots\alpha_r$, where the right-hand side is the
intersection product of cycles as defined in \cite[Chapter 8]{Ful}.
Topologically this homomorphism can be described by capping with the
orientation class in $H^{2(r-1)n}(M^{(r)},M^{(r)}-M)$ given by the regular
embedding $\Delta:M\rightarrow M^{(r)}$.

Let $t$ be a regular holomorphic section of $E$. Hence the set of the
zeros of $t$, $Z(t)$, is a closed subvariety of $M^{(r)}$ of dimension
$nr-d$.

\begin{proposition}\label{t:1}
The Gysin morphism satisfies:
\[
\Delta^!\big(c^{Vir}(Z(t))\big)=
c\left(\left(TM|_{Z(\Delta^*t)}\right)^{\oplus r-1}\right)
\cap c^{Vir}(Z(\Delta^*t)).
\]
\end{proposition}
\begin{proof}
By definition of the virtual class we have
\[
\Delta^!c^{Vir}(Z(t))=\Delta^!\left(
c\big(TM^{(r)}|_{Z(t)}\big)\cdot c\big(E|_{Z(t)}\big)^{-1}
\cap[Z(t)]\right).
\]
Hence, by \cite[Proposition 6.3]{Ful}, we have that
\[
\Delta^!c^{Vir}(Z(t))=
c\big(\Delta^*TM^{(r)}|_{Z(t)}\big)\cdot
c\big(\Delta^*E|_{Z(t)}\big)^{-1}\cap\Delta^![Z(t)].
\]
Note that $\Delta^*E|_{Z(t)}=\Delta^*E|_{Z(\Delta^*t)}$ and, by
\cite[Proposition 14.1]{Ful}, $\Delta^![Z(t)]=[Z(\Delta^*t)]$.
Moreover, since $\Delta^*TM^{(r)}=TM\oplus\ldots\oplus TM$, we have
\[
c\big(\Delta^*TM^{(r)}|_{Z(t)}\big)=
c\left(\left(TM|_{Z(\Delta^*t)}\right)^{\oplus r}\right).
\]
Then,
\begin{align*}
\Delta^!c^{Vir}(Z(t))
&=c\left(\left(TM|_{Z(\Delta^*t)}\right)^{\oplus r}\right)
\cdot c\big(\Delta^*E|_{Z(\Delta^*t)}\big)^{-1}\cap[Z(\Delta^*t)]\\
&=c\left(\left(TM|_{Z(\Delta^*t)}\right)^{\oplus r-1}\right)
\cap c^{Vir}(Z(\Delta^*t)).
\end{align*}
\end{proof}

For a subvariety $X$ of $M$, we denote its conormal variety by
$T_X^*M$, that is,
\[
T_X^*M:=\operatorname{closure}\{(x,\theta)\in T^*M\mid
x\in X_{reg}\text{ and }\theta|_{T_xX_{reg}}\equiv0\},
\]
where $T^*M$ is the cotangent space of $M$ and $X_{reg}$ is the
regular part of $X$. Let $L(M)$ be the free abelian group of all
cycles generated by the conormal spaces $T_X^*M$, where $X$ varies
over all subvarieties of $M$.

Let $\{S_\alpha\}$ be a Whitney stratification of $M$ adapted to $X$
and $x\in S_\alpha$ a point in $X$. Let $F(M)$ be the free abelian group
of constructible functions on $M$. Define the function
$Ch:F(M)\rightarrow L(M)$ by
\[
Ch(\xi):=\sum_\alpha(-1)^{\dim S_\alpha}
\eta(S_\alpha,\xi)\cdot T^*_{S_\alpha}M.
\]
Now consider the projectivized cotangent bundles $\mathbb{P}(T^*M)$
and $\mathbb{P}(T^*(M^{(r)}))$; we set
$\mathbb{P}(T^*M\oplus\ldots\oplus T^*M):=
\mathbb{P}((T^*M)^{\oplus r})$. Notice that one has a fibre square
diagram (see \cite[pag. 428]{Ful}):
\[
\xymatrix{
\mathbb{P}((T^*M)^{\oplus r})\ar[r]^{\delta}\ar[d]_{p}
&\mathbb{P}(T^*(M^{(r)}))\ar[d]^{\pi^{(r)}}\\
M\ar[r]^{\Delta}&M^{(r)}.
}
\]
where $\pi^{(r)}$ is the natural proper map. Let
$i:\mathbb{P}(T^*M)\rightarrow\mathbb{P}((T^*M)^{\oplus r})$ be the
morphism induced by the diagonal embedding
$T^*M\rightarrow T^*M\oplus\ldots\oplus T^*M$.

\begin{lemma}\label{L:4}
Let $Z(t)$ be as in Proposition \ref{t:1}. Assume that $Z(t)$ admits
a Whitney stratification $\{S_\alpha\}$ transversal to $\Delta(M)$.
Then:
\[
\delta^!\left[\mathbb{P}(Ch(\mathbbm{1}_{Z(t)}))\right]
=(-1)^{nr-n}i_*\left[\mathbb{P}(Ch(\mathbbm{1}_{Z(\Delta^*t)}))\right],
\]
where $\mathbbm{1}_{(\,)}$ denotes the characteristic function.
\end{lemma}
\begin{proof}
Since the stratification $\{S_\alpha\}$ is transversal to
$\Delta(M)$, we have that $\{\Delta^{-1}(S_\alpha)\}$ is a Whitney
stratification of $Z(\Delta^*t)$. By definition,
\[
\mathbb{P}(Ch(\mathbbm{1}_{Z(t)}))=
\sum_\alpha m_\alpha\mathbb{P}(T^*_{S_\alpha}M^{(r)}),
\]
where $m_\alpha:=(-1)^{nr-d-1}\chi(\phi_{f|Z(t)}F^\bullet)_z$,
with $z\in S_\alpha$, $F^\bullet$ the complex of sheaves defined by
$\chi(F^\bullet)_p=\mathbbm{1}_{Z(t)}(p)$, a germ
$f:(M^{(r)},z)\rightarrow(\C,0)$ such that satisfies
$(z,d_zf)\notin T^*_{S_j}M^{(r)}$ for all strata
$S_\beta\neq S_\alpha$ and $\phi_{f|Z(t)}F^\bullet$ is the sheaf
of vanishing cycles of $f$ restricted to $Z(t)$. Analogously,
\[
\mathbb{P}(Ch(\mathbbm{1}_{Z(\Delta^*t)}))=
\sum_\alpha n_\alpha\mathbb{P}(T^*_{\Delta^{-1}(S_\alpha)}M),
\]
where $n_\alpha:=(-1)^{n-d-1}\chi(\phi_{g|Z(\Delta^*t)}G^\bullet)_x$,
with $x\in\Delta^{-1}(S_\alpha)$, $G^\bullet$ the complex of sheaves
defined by $\chi(G^\bullet)_q=\mathbbm{1}_{Z(\Delta^*t)}(q)$ and a
germ $g:(M,x)\rightarrow(\C,0)$ such that satisfies
$(x,d_xg)\notin T^*_{\Delta^{-1}(S_\beta)}M$ for all strata
$S_\beta\neq S_\alpha$.
These definitions do not depend on the choices of $(z,f)$ and
$(x,g)$ respectively. So we fix $x$ and $g$ in the second definition,
and take $z=(x,\ldots,x)$ and $f$ such that $\Delta^*f=g$.
Note that, since $\Delta^*F^\bullet=G^\bullet$ and
\[
\phi_{\Delta^*(f|Z(t))}\Delta^*F^\bullet
=\Delta^*\phi_{f|Z(t)}F^\bullet,
\]
we have
\[
\chi(\phi_{g|Z(\Delta^*t)}G^\bullet)_x
=\chi(\phi_{f|Z(t)}F^\bullet)_{(x,\ldots,x)}.
\]
Hence
\begin{equation}\label{6}m_\alpha=(-1)^{nr-n}n_\alpha.\end{equation}
Notice that
\begin{equation}\label{7}
\delta^![\mathbb{P}(T^*_{S_i}M^{(r)})]
=i_*[\mathbb{P}(T^*_{\Delta^{-1}(S_i)}M)].
\end{equation}
Therefore, by equations (\ref{6}) and (\ref{7}),
\[
\delta^![\mathbb{P}(Ch(\mathbbm{1}_{Z(t)}))]
=(-1)^{nr-n}i_*[\mathbb{P}(Ch(\mathbbm{1}_{Z(\Delta^*t)}))].
\]
\end{proof}

\begin{theorem}\label{T:2}
With the assumptions of Lemma \ref{L:4} we have:
\[
\Delta^!c^{SM}(Z(t))=
c\left(\left(TM|_{Z(\Delta^*t)}\right)^{\oplus r-1}\right)
\cap c^{SM}(Z(\Delta^*t)).
\]
\end{theorem}
\begin{proof}
By the description of Schwartz-MacPherson classes due to C. Sabbah
in \cite{Sabbah}, we have that
\[
c^{SM}(Z(t))=(-1)^{nr-1}c(TM^{(r)}|_{Z(t)})\cap
\pi_*^{(r)}\big(c({\cal O}_r(1)^{-1})\cap
[\mathbb{P}(Ch(\mathbbm{1}_{Z(t)}))]\big),
\]
where ${\cal O}_r(1)$ denotes the tautological line bundle on the
projectivisation $\mathbb{P}(T^*M^{(r)})\rightarrow M^{(r)}$.
Using again \cite[Proposition 6.3]{Ful} we get:
\begin{equation}\label{8}
\Delta^!c^{SM}(Z(t))=(-1)^{nr-1}
c(\Delta^*TM^{(r)}|_{Z(t)})\cap
\Delta^!\pi_*^{(r)}\big(c({\cal O}_r(1))^{-1}\cap
[\mathbb{P}(Ch(\mathbbm{1}_{Z(t)}))]\big).
\end{equation}
And by \cite[Theorem 6.2]{Ful} we have,
\begin{equation}\label{9}
\Delta^!\pi_*^{(r)}\big(c({\cal O}_r(1))^{-1}\cap
[\mathbb{P}(Ch(\mathbbm{1}_{Z(t)}))]\big)
=p_*\big(c(\delta^*{\cal O}_r(1))^{-1}\cap
\delta^![\mathbb{P}(Ch(\mathbbm{1}_{Z(t)}))]\big).
\end{equation}
Since $\delta^*{\cal O}_r(1)={\cal O}_{\mathbb{P}((T^*M)^{\oplus r})}(1)$
is the tautological line bundle on the projectivisation
$\mathbb{P}((T^*M^{(r)})^{\oplus r})\rightarrow M^{(r)}$, by
Lemma \ref{L:4} and equations (\ref{8}) and (\ref{9}), we get:
\begin{align*}
\Delta^!c^{SM}(Z(t))
&=(-1)^{n-1}c\left((TM|_{Z(\Delta^*t)})^{\oplus r}\right)\cap\\
&\quad\cap p_*\big(c({\cal O}_{\mathbb{P}((T^*M)^{\oplus r})}(1))^{-1}
\cap i_*[\mathbb{P}(Ch(\mathbbm{1}_{Z(\Delta^*t)}))]\big).
\end{align*}
Using again the description of Schwartz-MacPherson due to Sabbah
({\it loc. cit.}) we obtain:
\[
\Delta^!c^{SM}(Z(t))=
c\left((TM|_{Z(\Delta^*t)})^{\oplus r-1}\right)
\cap c^{SM}(Z(\Delta^*t)).
\]
\end{proof}

\begin{corollary}\label{10}
With the previous assumptions we have:
\[
\Delta^!{\cal M}(Z(t))=(-1)^{nr-n}
c\left((TM|_{Z(\Delta^*t)})^{\oplus r-1}\right)
\cap{\cal M}(Z(\Delta^*t)).
\]
\end{corollary}
\begin{proof}
Using Proposition \ref{t:1} and Theorem \ref{T:2}, we have:
\begin{align*}
\Delta^!{\cal M}(Z(t))
&=\Delta^!\big((-1)^{nr-d}(c^{Vir}(Z(t))-c^{SM}(Z(t)))\big)\\
&=(-1)^{nr-d}c\left((TM|_{Z(\Delta^*t)})^{\oplus r-1}\right)
(c^{Vir}(Z(\Delta^*t))-c^{SM}(Z(\Delta^*t)))\\
&=(-1)^{nr-n}c\left((TM|_{Z(\Delta^*t)})^{\oplus r-1}\right)
\cap{\cal M}(Z(\Delta^*t)).
\end{align*}
\end{proof}

\section{Intersection product formula}

As before, let $M$ be an $n$-dimensional compact complex analytic manifold.
Let $\{E_i\}$ be a finite collection of holomorphic vector bundles over $M$
of rank $d_i$, $2\leq i\leq r$. For each of these bundles, let
$s_i:M\rightarrow E_i$ be a regular holomorphic section and $X_i$ the
$(n-d_i)$-dimensional local complete intersections defined by the zeroes
of $s_i$. In this section we assume that the $X_i$ are equipped with
Whitney stratifications such that all the intersections amongst strata
in the various $X_i$ are transversal.

Let $p_i:M^{(r)}\rightarrow M$ be the $i^{th}$-projection, then we have
the holomorphic exterior product section
\[
s=s_1\oplus\ldots\oplus s_r:M^{(r)}\rightarrow
p_1^*E_1\oplus\ldots\oplus p_r^*E_r,
\]
given by $s(x_1,\ldots,x_r)=(s_1(x_1),\ldots,s_r(x_r))$.
Then $Z(s)=X_1\times\ldots\times X_r$ and
$Z(\Delta^*(s))=X_1\cap\ldots\cap X_r$. Set $X=Z(\Delta^*(s))$.

The next result describes the total Milnor class of $X$ in terms of
the total Schwartz-MacPherson and Milnor classes of the $X_i$.

\begin{theorem} \label{4.1}
 In the above conditions we have:
 $${\cal M}(X)=(-1)^{nr-n}c\left( \left(
TM|_{X}\right)^{\oplus r-1} \right)^{-1}\cap
\sum\;(-1)^{(n-d_{1})\epsilon_{1}+\ldots+(n-d_{r})\epsilon_{r}}
 P_{1}\;\cdot\ldots\cdot\; P_{r}\in A_{*}(X),$$
where the sum runs over all choices  of $P_{i}\in \left\{{\cal
M}(X_i),c^{SM}(X_i)\right\},\,i=1,\dots,r,$  except
$(P_{1},\ldots,P_{r})=(c^{SM}(X_1),\ldots,c^{SM}(X_r))$ and where
$$\epsilon_{i}=\left\{\begin{array}{rcl} 1&,& \mbox{if} \;P_{i}=c^{SM}(X_i)\\
0&,& if \;P_{i}={\cal M}(X_i)\\
\end{array}\right. .$$

\end{theorem}

Thus for instance, when $r=2$ we have:
$${\cal M}(X)=c\left( \left( TM|_{X}\right) \right)^{-1}\cap
\qquad \qquad\qquad\qquad   \qquad \qquad\qquad\qquad  \qquad
\qquad\qquad\qquad  $$
$$\quad \Big((-1)^{n}{\cal
M}(X_1)\cdot {\cal M}(X_2)+ (-1)^{d_{1}} c^{SM}(X_1)\cdot  {\cal
M}(X_2)+(-1)^{d_{2}}{\cal M}(X_1)\cdot  c^{SM}(X_2)\Big).$$
For $r= 3$ : {\footnotesize
$${\cal M}(X)= c\left( \left( TM|_{X}\right)^{\oplus 2}
\right)^{-1}\cap \;\; \Big({\cal M}(X_1)\cdot {\cal M}(X_2)\cdot
{\cal M}(X_3)+ (-1)^{(d_{1}+d_{2})} c^{SM}(X_1)\cdot
c^{SM}(X_2)\cdot  {\cal M}(X_3) +$$
$$+(-1)^{(d_{1}+d_{3})} c^{SM}(X_1)\cdot {\cal
M}(X_2)\cdot  c^{SM}(X_3)+(-1)^{(d_{2}+d_{3})} {\cal M}(X_1)\cdot
c^{SM}(X_2)\cdot  c^{SM}(X_3)+$$
$$+(-1)^{(n-d_{1})} c^{SM}(X_1)\cdot {\cal M}(X_2)\cdot {\cal
M}(X_3)+(-1)^{(n-d_{2})} {\cal M}(X_1)\cdot c^{SM}(X_2)\cdot {\cal
M}(X_3)+$$
$$+(-1)^{(n-d_{3})} {\cal M}(X_1)\cdot {\cal M}(X_2)\cdot  c^{SM}(X_3)
\Big),$$} and so on.

\begin{proof}
By Corollary \ref{10},
$$\Delta^{!}{\cal M}(Z(s))=(-1)^{nr-n}c\left( \left( TM|_{Z(\Delta^{*}s)}\right)^{\oplus
r-1} \right)\cap{\cal M}(Z(\Delta^{*}s)).$$ Thus,
$${\cal M}(X)=(-1)^{nr-n}c\left( \left(
TM|_{Z(\Delta^{*}s)}\right)^{\oplus r-1}
\right)^{-1}\cap\Delta^{!}{\cal M}(X_{1}\times \ldots \times
X_{r}),$$
 and using the description of the Milnor classes of a product due to   \cite[Theorem 3.3]{O-Y}, we have:
$${\cal M}(X)=(-1)^{nr-n}c\left( \left(
TM|_{X}\right)^{\oplus r-1} \right)^{-1}\cap
\;\sum(-1)^{(n-d_{1})\epsilon_{1}+\ldots+(n-d_{r})\epsilon_{r}}\Delta^{!}\left(
 P_{1}\times \ldots\times  P_{r}\right),$$
where the sum runs over all choices  of $P_{i}\in \left\{{\cal
M}(X_i),c^{SM}(X_i)\right\},\,i=1,\dots,r,$  except
$(P_{1},\ldots,P_{r})=(c^{SM}(X_1),\ldots,c^{SM}(X_r))$ and
where $$\epsilon_{i}=\left\{\begin{array}{rcl} 1&,& if \; P_{i}=c^{SM}(X_i)\\
0&,& if \; P_{i}={\cal M}(X_i)\\
\end{array}\right. .$$
Therefore the result follows because  $$\Delta^{!}\left(
P_{1}\times \ldots\times P_{r}\right)= P_{1}\;\cdot\ldots\cdot\;
P_{r}\in A_{*}(X).$$
\end{proof}

Using now that the total Milnor class is, up to sign, the difference
between the total Schwartz-MacPherson and the total virtual class we obtain:

\begin{corollary}\label{11}
\[
{\cal M}(X)=(-1)^{r-1}
c\left(\left(TM|_X\right)^{\oplus r-1}\right)^{-1}\cap
\sum_{i=1}^r a_{1,i}\cdots a_{r-1,i}\cdot{\cal M}(X_i),
\]
where
\[
a_{j,i}=\begin{cases}
c^{Vir}(X_{j+1}) &\text{if }j\leq i,\\
c^{SM}(X_j) &\text{if }j>i.
\end{cases}
\]
\end{corollary}

\begin{corollary}
\[
{\cal M}(X)=(-1)^{n-r}
c\left(\left(TM|_X\right)^{\oplus r-1}\right)^{-1}\cap
\left(c^{Vir}(X_1)\cdots c^{Vir}(X_r)
-c^{SM}(X_1)\cdots c^{SM}(X_r)\right).
\]
\end{corollary}

\section{Applications to line bundles}

In this section we replace the bundles $E_i$ by line bundles $L_i$.
We assume that each hypersurface $X_i$ is equipped with a Whitney
stratification ${\cal S}_i$ such that all the intersections amongst
strata in the various $X_i$ are transversal. We set
$X:=X_1\cap\ldots\cap X_r$ and we obtain in this section three
applications of the formulas in the previous section.

\subsection{Parusi\'{n}ski-Pragacz-type formula}
Now we extend to local complete intersections as above the following
important characterization of Milnor classes obtained by A.
Parusi\'{n}ski and P. Pragacz in \cite{Par-Pr} for hypersurfaces in
compact manifolds:
\begin{equation}\label{P}
{\cal M}(X_i):=\sum_{S\in{\cal S}_i}\gamma_S
\big(c(L_{i|_{X_i}})^{-1}\cap c^{SM}(S)\big)\in A_*(X_i),
\end{equation}
where $\gamma_S$ is the function defined on each stratum $S$ as follows:
For each $x\in S\subset X_i$, let $F_x$ be a local Milnor fibre
(recall $X_i$ is a hypersurface in the complex manifold $M$), and let
$\chi(F_x)$ be its Euler characteristic. We set:
\[
\mu(x;X_i):=(-1)^n(\chi(F_x)-1),
\]
and call it the local Milnor number of $X_i$ at $x$. This number is
constant on each Whitney stratum, by the topological triviality of
Whitney stratifications, so we denote it $\mu_S$. Then $\gamma_S$
is defined inductively by:
\[
\gamma_S=\mu_S-\sum_{S'\neq S,\;\overline{S'}\supset S}\gamma_{S'}.
\]
Following the idea of Parusi\'{n}ski and Pragacz of describing the
Milnor classes using only the Schwartz-MacPherson classes, we have
as a consequence of Corollary \ref{11} and equation (\ref{P}) the
following:
\begin{corollary}
The total Milnor class of $X$ is:
\[
\begin{split}
{\cal M}(X)=&(-1)^{r-1}
c\left((TM|_X)^{\oplus r-1}\right)^{-1}\cap\\
&\cap\sum_{i=1}^r\sum_{S\in{\cal S}_i}\gamma_S
\,a_{1,i}\cdots a_{r-1,i}\cdot
\big(c(L_{i|_{X_i}})^{-1}\cap c^{SM}(S)\big),
\end{split}
\]
where
\[
a_{j,i}=\begin{cases}
c^{Vir}(X_{j+1})&\text{if }j\leq i,\\
c^{SM}(X_j)&\text{if }j>i.
\end{cases}
\]
\end{corollary}

Using \cite[Example 3.2.8 and Proposition 6.3]{Ful} we have:
\begin{corollary}[Parusi\'{n}ski-Pragacz formula for local complete intersections]
\label{P-P}
In the above conditions we have:
\[
\begin{split}
{\cal M}(X)=&(-1)^{nr-n}
c\left((TM|_X)^{\oplus r-1}\right)^{-1}\cap\\
&\cap\sum\alpha_{S_1,\ldots,S_r}^{\epsilon_1,\ldots,\epsilon_r}
\frac{c(L_1)^{\epsilon_1}\cdots c(L_r)^{\epsilon_r}}
{c(L_1\oplus\ldots\oplus L_r)}\cap
c^{SM}(S_1)\cdots c^{SM}(S_r),
\end{split}
\]
where the sum runs over all choices of $S_i\in{\cal S}_i$,
$i=1,\ldots,r$, except $(S_1,\ldots,S_r)=
((X_1)_{reg},\ldots,(X_r)_{reg})$ and where $(X_i)_{reg}$
denotes the regular part of $X_i$,
\[
\alpha_{S_1,\ldots,S_r}^{\epsilon_1,\ldots,\epsilon_r}
=(-1)^{(n-1)(\epsilon_1+\ldots+\epsilon_r)}
\gamma_{S_1}^{1-\epsilon_1}\cdots\gamma_{S_r}^{1-\epsilon_r},
\]
and
\[
\epsilon_i=\begin{cases}
1&\text{if }S_i=(X_i)_{reg},\\
0&\text{if }\dim(S_i)<n-1.
\end{cases}
\]
\end{corollary}

\subsection{Milnor classes and L\^e cycles}
Let ${\rm Sing}(X_i)$ be the singular set of $X_i$. This is the set
of points where the section $s_i$ fails to be transversal to the
zero-section of $L_i$. Consider the blow-up
${\rm Bl}_{{\rm Sing}(X_i)}M:=B_i$ of $M$ along ${\rm Sing}(X_i)$,
let $D_i$ be the exceptional divisor of $B_i$ and $\mathcal L_i$ the
associated line bundle on $B_i$, that we call the tautological line
bundle of $B_i$. One has a diagram:
\[
\xymatrix{
\mathcal L_i\ar[d] & \\
D_i\ar[r]\ar[d] & B_i\ar[d]^{b_i}\\
{\rm Sing}(X_i)\ar[r] & M.
}
\]
\begin{definition}
For $0\leq k\leq n$, we define the $k$-th L\^e cycle of the section
$s_i$ as the following class in the Chow group of $M$:
\[
\Lambda_k(X_i)=(b_i)_*\big(c_1(\mathcal L_i)^{n-k}\cap[D_i]\big).
\]
\end{definition}
Notice that $\Lambda_k(X_i)$ is supported in ${\rm Sing}(X_i)$,
hence we can also think of $\Lambda_k(X_i)$ as being a class in the
Chow group of ${\rm Sing}(X_i)$. Notice also that $\Lambda_k(X_i)$
is zero for $k>\dim {\rm Sing}(X_i)$.
We point out that the cycles obtained in this way are special cases
of Segre cycles (see for instance \cite{Ful}). Moreover, because
$M$ is compact, these represent classes in the homology ring
$H_*(M;\Z)$.

In \cite[Theorem 1]{BMS} the authors proved the following description
for the Milnor classes of $X_i$ in terms of the L\^e cycles of $X_i$:
\[
{\cal M}_k(X_i)=\sum_{l=0}^{d-k}(-1)^{k+l}
\binom{l+k}{k}c_1(L_i|_{X_i})^l\cap\Lambda_{l+k}(X_i).
\]
Based on this formula and that of Corollary \ref{11}, we get a
description of Milnor classes of local complete intersections
$X=X_1\cap\ldots\cap X_r$ via the L\^e cycles of each
hypersurface $X_i$:
\begin{corollary}\label{leformula}
For each $0\leq k\leq\dim X$, the $k^{th}$ Milnor class of $X$ is:
\[
\begin{split}
{\cal M}_k(X)=&(-1)^{r-1}
c\left((TM|_X)^{\oplus r-1}\right)^{-1}\cap\\
&\cap\sum_{i=1}^r\sum_{l=0}^{d_i-k}(-1)^{k+l}
\binom{l+k}{k}(a_{1,i})\cdots(a_{r-1,i})\cdot
\big(c_1(L_i|_{X_i})^l\cap\Lambda_{l+k}(X_i)\big),
\end{split}
\]
where
\[
a_{j,i}=\begin{cases}
c^{Vir}(X_{j+1})&\text{if }j\leq i,\\
c^{SM}(X_j)&\text{if }j>i.
\end{cases}
\]
\end{corollary}

\subsection{Aluffi type formula}
Let $X=X_1\cap\ldots\cap X_r$ as above. Now we express the total
Milnor class of $X$ in terms of the Aluffi's $\mu$-class of each
hypersurface $X_i$. The $\mu$-class was introduced in \cite{Aluffi}
and it involves the Segre class of the singular locus of the
hypersurface in the total smooth ambient variety.

For each hypersurface $X_i$ of $M$, the Aluffi's $\mu$-class
$\mu_{L_i}({\rm Sing}(X_i))$ of the singular locus of $X_i$
is defined by the formula
\[
\mu_{L_i}({\rm Sing}(X_i))=
c(T^*M\otimes L_i)\cap s({\rm Sing}(X_i),M),
\]
where $s({\rm Sing}(X_i),M)$ is the Segre class of
${\rm Sing}(X_i)$ in $M$ (see \cite[Chapter 4]{Ful}).
We need to introduce some notation. If $\alpha\in A_*(X_i)$ is a
cycle in the Chow group of $X_i$ and $\alpha=\sum_{j\geq0}\alpha_j$,
where $\alpha_j$ is the codimension $j$ component of $\alpha$, then
Aluffi introduced the following cycles
\[
\alpha^{\vee}:=\sum_{j\geq0}(-1)^j\alpha_j,
\qquad
\alpha\otimes L_i:=\sum_{j\geq0}\frac{\alpha_j}{c(L_i)^j}.
\]
Then Aluffi proved in \cite{Aluffi} that the total Milnor class
${\cal M}(X_i)$ can be described as follows:
\begin{equation}\label{Aluffi-formula}
{\cal M}(X_i)=(-1)^{n-1}c(L_i)^{n-1}\cap
(\mu_{L_i}({\rm Sing}(X_i))^{\vee}\otimes L_i).
\end{equation}
Again using Corollary \ref{11}, the above equation yields:
\begin{corollary}\label{Aluffi}
The Total Milnor class of $X:=X_1\cap\ldots\cap X_r$ is:
\[
\begin{split}
{\cal M}(X)=&(-1)^{n-r}
c\left((TM|_X)^{\oplus r-1}\right)^{-1}\cap\\
&\cap\sum_{i=1}^r a_{1,i}\cdots a_{r-1,i}\cdot
c(L_i)^{n-1}\cap
\big(\mu_{L_i}({\rm Sing}(X_i))^{\vee}\otimes L_i\big),
\end{split}
\]
where
\[
a_{j,i}=\begin{cases}
c^{Vir}(X_{j+1})&\text{if }j\leq i,\\
c^{SM}(X_j)&\text{if }j>i.
\end{cases}
\]
\end{corollary}

\section{General Case}

Let $M$ be an $m$-dimensional compact complex analytic manifold and let
$E$ be a holomorphic vector bundle over $M$ of rank $r$. Let $Z(s)$ be
the zero set of a regular holomorphic section $s$ of $E$, so this is
an $(m-r)$-dimensional local complete intersection.

Let $p:N\rightarrow M$ be a proper morphism, where $N$ is an
$n$-dimensional compact complex analytic manifold. Consider the diagram
\[
\xymatrix{
\mathbb{P}(T^*N) &
\mathbb{P}(p^*(T^*M))\ar[l]_{i}\ar[r]^{g}
\ar[d]^{p^*(\pi_M)} &\mathbb{P}(T^*M)\ar[d]^{\pi_M}\\
&N\ar[r]^{p}&M
}
\]
where $\pi_M:\mathbb{P}(T^*M)\rightarrow M$ and
$\pi_N:\mathbb{P}(T^*N)\rightarrow N$ are the corresponding
projectivized cotangent bundles.

\begin{lemma}
Let $Ch(\mathbbm{1}_{Z(s)})$ be the characteristic cycle of the
characteristic function $\mathbbm{1}_{Z(s)}$. Then
\[
i_*g^*\big[\mathbb{P}(Ch(\mathbbm{1}_{Z(s)}))\big]
=(-1)^{n-m}\big[\mathbb{P}(Ch(\mathbbm{1}_{Z(p^*s)}))\big].
\]
\end{lemma}
\begin{proof}
Let $\{S_\alpha\}$ be a Whitney stratification of $Z(s)$. Then
$\{p^{-1}(S_\alpha)\}$ is a Whitney stratification of $Z(p^*s)$.
By definition,
\[
\mathbb{P}(Ch(\mathbbm{1}_{Z(s)}))=
\sum_\alpha m_\alpha\mathbb{P}(T^*_{S_\alpha}M),
\]
where $m_\alpha:=(-1)^{m-r-1}
\chi(\phi_{h_1|Z(s)}H_1^\bullet)_{z_1}$ with $z_1\in S_\alpha$,
$H_1^\bullet$ is the complex of sheaves such that
$\chi(H_1^\bullet)_z=\mathbbm{1}_{Z(s)}(z)$,
$h_1:(M,z_1)\rightarrow(\C,0)$ is a germ that satisfies
$(z,d_zh_1)\in T^*_{S_\alpha}M$ and
$(z,d_zh_1)\notin T^*_{S_\beta}M$ for all strata
$S_\beta\neq S_\alpha$ and $\phi_{h_1|Z(s)}H_1^\bullet$
is the sheaf of vanishing cycles of $h_1$ restricted to $Z(s)$.
Analogously,
\[
\mathbb{P}(Ch(\mathbbm{1}_{Z(p^*s)}))=
\sum_\alpha n_\alpha\mathbb{P}(T^*_{p^{-1}(S_\alpha)}N),
\]
where $n_\alpha:=(-1)^{n-r-1}
\chi(\phi_{h_2|Z(p^*s)}H_2^\bullet)_{z_2}$ with
$z_2\in p^{-1}(S_\alpha)$, $H_2^\bullet$ is the complex of
sheaves such that $\chi(H_2^\bullet)_q=
\mathbbm{1}_{Z(p^*s)}(q)$, $h_2:(N,z_2)\rightarrow(\C,0)$ is a
germ that satisfies
$(z_2,d_{z_2}h_2)\in T^*_{p^{-1}(S_\alpha)}N$ and
$(z_2,d_{z_2}h_2)\notin T^*_{p^{-1}(S_\beta)}N$ for all
strata $S_\beta\neq S_\alpha$.
These definitions do not depend on the choices of $(z_1,h_1)$ and
$(z_2,h_2)$ respectively. So we first choose $z_1$ and $h_1$ in
the first definition, then fix a $z_2\in p^{-1}(z_1)$ and set
$h_2=h_1\circ p$.
Note that
\[
p^*H_1^\bullet=H_2^\bullet\quad\text{and}\quad
\phi_{p^*(h_1|Z(s))}p^*H_1^\bullet
=p^*\phi_{h_1|Z(s)}H_1^\bullet.
\]
Thus we have that
$\chi(\phi_{h_2|Z(p^*s)}H_2^\bullet)_{z_2}
=\chi(\phi_{h_1|Z(s)}H_1^\bullet)_{z_1}$. Hence
\begin{equation}\label{eq-general-9}
m_\alpha=(-1)^{n-m}n_\alpha.
\end{equation}
Note that
\begin{equation}\label{eq-general-10}
i_*g^*[\mathbb{P}(T^*_{S_\alpha}M)]
=[\mathbb{P}(T^*_{p^{-1}(S_\alpha)}N)].
\end{equation}
Therefore, by equations (\ref{eq-general-9}) and (\ref{eq-general-10}),
we get
\[
i_*g^*[\mathbb{P}(Ch(\mathbbm{1}_{Z(s)}))]
=(-1)^{n-m}[\mathbb{P}(Ch(\mathbbm{1}_{Z(p^*s)}))],
\]
as stated.
\end{proof}

\begin{lemma}\label{1}
In the conditions above, if the morphism $p:N\rightarrow M$ is also
flat, then
\[
p^*({\cal M}(Z(s)))=
\frac{c(p^*(TM))}{c(TN)}\cap{\cal M}(Z(p^*s)).
\]
\end{lemma}
\begin{proof}
Since $p$ is flat, $p^*[Z(s)]=[Z(p^*s)]$ (by
\cite[Proposition 14.1]{Ful}). Hence,
\begin{align*}
p^*c^{Vir}(Z(s))
&=p^*\big(c(TM)c(E)^{-1}\cap[Z(s)]\big)\\
&=c(p^*TM)c(p^*E)^{-1}\cap p^*[Z(s)]\\
&=\frac{c(p^*TM)}{c(TN)}c(TN)c(p^*E)^{-1}
\cap[Z(p^*s)]\\
&=\frac{c(p^*TM)}{c(TN)}\cap c^{Vir}(Z(p^*s)).
\end{align*}
On the other hand, using the description of Schwartz-MacPherson
classes due to C. Sabbah in \cite{Sabbah}, we have:
\[
\begin{split}
p^*c^{SM}(Z(s))=&(-1)^{m-1}\frac{c(p^*TM)}{c(TN)}c(TN)\cap\\
&\cap p^*\pi_{M*}\big(c({\cal O}_M(1))^{-1}\cap
[\mathbb{P}(Ch(\mathbbm{1}_{Z(s)}))]\big),
\end{split}
\]
where ${\cal O}_M(1)$ denotes the tautological line bundle on
the projectivized cotangent bundle
$\pi_M:\mathbb{P}(T^*M)\rightarrow M$. By Lemma 6.1, we have:
\[
\begin{split}
p^*\pi_{M*}\big(c({\cal O}_M(1))^{-1}\cap
[\mathbb{P}(Ch(\mathbbm{1}_{Z(s)}))]\big)
=&(-1)^{n+m}\pi_{N*}\big(c({\cal O}_N(1))^{-1}\cap\\
&\quad\cap[\mathbb{P}(Ch(\mathbbm{1}_{Z(p^*s)}))]\big),
\end{split}
\]
where ${\cal O}_N(1)$ is the tautological line bundle on the
projectivisation $\pi_N:\mathbb{P}(T^*N)\rightarrow N$. Hence,
\begin{align*}
p^*c^{SM}(Z(s))
&=\frac{c(p^*TM)}{c(TN)}\cap
\Big((-1)^{n-1}c(TN)\cap
\pi_{N*}\big(c({\cal O}_N(1))^{-1}\cap
[\mathbb{P}(Ch(\mathbbm{1}_{Z(p^*s)}))]\big)\Big)\\
&=\frac{c(p^*TM)}{c(TN)}\cap c^{SM}(Z(p^*s)).
\end{align*}
Therefore:
\[
p^*({\cal M}(Z(s)))=
\frac{c(p^*(TM))}{c(TN)}\cap{\cal M}(Z(p^*s)).
\]
\end{proof}

\begin{lemma}\label{2}
Consider an exact sequence of vector bundles on $M$:
\[
0\longrightarrow F\longrightarrow G\longrightarrow H\longrightarrow0.
\]
Given a section $t$ of $G$, let $t_H$ be the induced section of $H$.
Set $Z_1=Z(t_H)$ and let $\tilde{t}$ be the section of $F|_{Z_1}$
induced by $t|_{Z_1}$. Then,
\[
{\cal M}(Z(t))=c(F)^{-1}c_{top}(F)\cap{\cal M}(Z_1),
\]
where $c_{top}(F)$ is the top Chern class of $F$.
\end{lemma}
\begin{proof}
Using \cite[Proposition 1.3]{PP} we have:
\[
c^{SM}(Z(t))=c(F)^{-1}c_{top}(F)\cap c^{SM}(Z_1).
\]
On the other hand, let us consider the following virtual bundles:
\begin{align*}
\tau(Z_1,M)&=TM|_{Z_1}-H|_{Z_1},\\
\tau(Z(t),M)&=TM|_{Z(t)}-G|_{Z(t)},\\
\tau(Z(t),Z_1)&:=\tau(Z_1,M)-i^*F|_{Z(t)}.
\end{align*}
Notice that $c_{top}(i^*F)\cap[Z_1]=j_*[Z(\tilde{t})]$ where
$j:Z(\tilde{t})\rightarrow Z_1$ is the inclusion. On the other hand
$Z(\tilde{t})=Z(t)$. Hence we have:
\begin{align*}
c_{top}(i^*F)\cap c^{Vir}(Z_1,M)
&=c_{top}(i^*F)\cap c(TM)c(H)^{-1}\cap[Z_1]\\
&=c(TM)c(H)^{-1}\cap j_*[Z(t)].
\end{align*}
Then, setting $c^{Vir}(Z(t),Z_1):=c(\tau(Z(t),Z_1))$, we get:
\[
c(F)^{-1}c_{top}(F)\cap c^{Vir}(Z_1,M)
=c^{Vir}(Z(t),Z_1)=c^{Vir}(Z(t),M),
\]
{i.e.}, $c^{Vir}(Z(t))=c(F)^{-1}c_{top}(F)\cap c^{Vir}(Z_1)$.
Thence:
\[
{\cal M}(Z(t))=c(F)^{-1}c_{top}(F)\cap{\cal M}(Z_1),
\]
as stated.
\end{proof}

\begin{theorem}\label{t:final}
Let $s:M\rightarrow E$ be a regular section of the holomorphic
bundle $E$. Let $p:\mathbb{P}(E^{\vee})\rightarrow M$ be the
projectivization of the dual bundle $E^{\vee}$, so we have the
tautological exact sequence
$0\rightarrow F\rightarrow p^*E\rightarrow
{\cal O}_{\mathbb{P}(E^{\vee})}(1)\rightarrow0$.
Then,
\[
\begin{split}
{\cal M}(Z(s))=p_*\Big(&\big[c\big(p^*(E^{\vee})\otimes
{\cal O}_{\mathbb{P}(E^{\vee})}(1)\big)^{-1}\cdot\\
&\cdot c_1\big({\cal O}_{\mathbb{P}(E^{\vee})}(1)\big)^{r-1}
\cdot c(F)^{-1}\cdot c_{top}(F)\big]
\cap{\cal M}(Z(\tilde{s}))\Big),
\end{split}
\]
where $\tilde{s}$ is the section induced by $p^*s$ in
${\cal O}_{\mathbb{P}(E^{\vee})}(1)$.
\end{theorem}
\begin{proof}
By Lemma \ref{1},
\[
p^*({\cal M}(Z(s)))=\frac{c(p^*TM)}{c(T\mathbb{P}(E^{\vee}))}
\cap{\cal M}(Z(p^*s)).
\]
It follows from \cite[Proposition 3.1]{Ful} that
\[
{\cal M}(Z(s))=p_*\left(c_1({\cal O}_{\mathbb{P}(E^{\vee})}(1))^{r-1}
\frac{c(p^*TM)}{c(T\mathbb{P}(E^{\vee}))}
\cap{\cal M}(Z(p^*s))\right).
\]
Using Lemma \ref{2} we have:
\[
{\cal M}(Z(p^*s))=c(F)^{-1}c_{top}(F)\cap{\cal M}(Z(\tilde{s})),
\]
and we arrive to
\begin{equation}\label{eq-last-11}
\begin{split}
{\cal M}(Z(s))=p_*\Big(&c_1({\cal O}_{\mathbb{P}(E^{\vee})}(1))^{r-1}
\frac{c(p^*TM)}{c(T\mathbb{P}(E^{\vee}))}\\
&\cdot c(F)^{-1}c_{top}(F)\cap{\cal M}(Z(\tilde{s}))\Big).
\end{split}
\end{equation}
Now using the exact sequence
\[
0\rightarrow T_{\mathbb{P}(E^{\vee})/M}\rightarrow
T\mathbb{P}(E^{\vee})\rightarrow p^*(TM)\rightarrow0,
\]
we get
\begin{equation}\label{eq-last-12}
\frac{c(p^*(TM))}{c(T\mathbb{P}(E^{\vee}))}
=c(T_{\mathbb{P}(E^{\vee})/M})^{-1}.
\end{equation}
On the other hand, using the exact sequence
\[
0\rightarrow{\cal O}_{\mathbb{P}(E^{\vee})}\rightarrow
p^*(E^{\vee})\otimes{\cal O}_{\mathbb{P}(E^{\vee})}(1)
\rightarrow T_{\mathbb{P}(E^{\vee})/M}\rightarrow0,
\]
we get
\begin{equation}\label{eq-last-13}
c\big(p^*(E^{\vee})\otimes{\cal O}_{\mathbb{P}(E^{\vee})}(1)\big)
=c(T_{\mathbb{P}(E^{\vee})/M}).
\end{equation}
Therefore by equations (\ref{eq-last-11}), (\ref{eq-last-12}) and
(\ref{eq-last-13}) we obtain:
\[
\begin{split}
{\cal M}(Z(s))=p_*\Big(&\big[c\big(p^*(E^{\vee})\otimes
{\cal O}_{\mathbb{P}(E^{\vee})}(1)\big)^{-1}
\cdot c_1\big({\cal O}_{\mathbb{P}(E^{\vee})}(1)\big)^{r-1}\\
&\cdot c(F)^{-1}c_{top}(F)\big]\cap{\cal M}(Z(\tilde{s}))\Big).
\end{split}
\]
\end{proof}

\end{document}